\documentclass[10pt,a4paper]{article}
\usepackage[english]{babel}

\usepackage{amsmath, amsfonts, amssymb, graphicx, amsthm, dsfont}
\usepackage{indentfirst}
\usepackage{xcolor}
\usepackage{xparse}

\usepackage{mathabx} % for \righttoleftarrow usage
\usepackage{mathrsfs}
\usepackage{enumitem}
\usepackage{authblk}

\newcommand{\toitself}[1]{#1\righttoleftarrow}
\theoremstyle{plain}
\newtheorem{theorem}{Theorem}[section]

\newtheorem{lemma}[theorem]{Lemma}
\newtheorem{prop}[theorem]{Proposition}
\newtheorem{corollary}[theorem]{Corollary}

\theoremstyle{definition}
\newtheorem{defi}[theorem]{Definition}
\newtheorem{remark}[theorem]{Remark}

\renewenvironment{proof}[1][.]{%
\bigskip\noindent{\bf Proof#1 }}{%
\hfill$\blacksquare$\bigskip}
\newcommand{\N}{\mathbb{N}}
\newcommand{\R}{\mathbb{R}}

\newcommand{\mcal}[1]{\mathcal{#1}}
\newcommand{\mscr}[1]{\mathscr{#1}}

\newcommand{\ds}{\displaystyle}
\newcommand{\q}{\text{q}\hspace{0.05em}}
\newcommand{\Px}[1][\q]{\,\text{P}_{x}^{\hspace*{0.03cm}#1}\hspace{0.05em}}
\newcommand{\Pxx}[2][\q]{\,\text{P}_{#2}^{\hspace*{0.03cm}#1}\hspace{0.05em}}

\newcommand{\dgamma}{\,\text{d}\gamma}
\newcommand{\dq}{\,\text{d}\q}
\newcommand{\dmu}{\,\text{d}\mu}

\newcommand{\dnu}{\,\text{d}\nu}
\newcommand{\dP}[1][\q]{\,\text{dP}_{x}^{\hspace*{0.03cm}#1}\hspace{0.01em}}
\newcommand{\dPx}[2][\q]{\,\text{dP}_{#2}^{\hspace*{0.03cm}#1}\hspace{0.01em}}

\newcommand{\X}{\textbf{X}}
\newcommand{\Y}{\textbf{Y}}

\newcommand{\HR}{\mcal{H}\left(\mcal{R}\right)}
\newcommand{\Rq}[1][\q]{\mcal{R}_{#1}}
\newcommand{\Lq}[1][\q]{\mcal{L}_{#1}}
\newcommand{\Bq}[1][\q]{\text{B}\hspace{0.05em}_{#1}}
\newcommand{\Bu}[1][\mu]{\text{B}\hspace{0.05em}_{#1}}

\newcommand{\Ruelle}{L\hspace{0.01em}}

\newcommand{\p}{\text{p}\hspace{0.05em}}
\newcommand{\dpp}{\,\text{d}\p}
\newcommand{\Rp}{\mcal{R}_{\p}}
\newcommand{\supqx}{\sup\q}
\newcommand{\infqx}{\inf\q}
\newcommand{\df}[1][\hspace{-0.15em}]{\text{d}_{#1}f}
\newcommand{\hv}[1][\mu]{h_{\text{v}}^{#1}}
\newcommand{\ha}[1][\mu]{h_{\text{a}}^{#1}}

\makeatletter
\def\moverlay{\mathpalette\mov@rlay}
\def\mov@rlay#1#2{\leavevmode\vtop{%
    \baselineskip\z@skip\lineskiplimit-\maxdimen%
    \ialign{\hfil$\m@th#1##$\hfil\cr#2\crcr}}}
\newcommand{\charfusion}[3][\mathord]{
    #1{\ifx#1\mathop\vphantom{#2}\fi
        \mathpalette\mov@rlay{#2\cr#3}
      }
    \ifx#1\mathop\expandafter\displaylimits\fi}
\makeatother

\newcommand\restr[2]{{% we make the whole thing an ordinary symbol
  \left.\kern-\nulldelimiterspace% automatically resize the bar with \right
  #1 % the function
  \vphantom{\big|} % pretend it's a little taller at normal size
  \right|_{#2} % this is the delimiter
  }}

\newcommand\norm[1]{\left\lVert#1\right\rVert}
\newcommand\abs[1]{\left\lvert#1\right\rvert}

\newcommand\intervalco[1]{[#1)}% chktex 8 chktex 9
\newcommand\email[1]{\,\footnotesize\texttt{#1}}
\title{Thermodynamic formalism for general iterated function systems with
measures}

\author[*]{Jader E. Brasil}
\author[$\dagger$]{Elismar R. Oliveira}
\author[$\ddagger$]{Rafael Rig\~ao Souza}
\affil[$\,$]{
    \footnotesize{Instituto de Matem\'atica e Estat\'istica - UFRGS}\\
    \footnotesize{Av. Bento Gon\c calves, 9500}\\
    \footnotesize 91509-900, Porto Alegre - RS, Brazil
    \vspace{10px}
}
\affil[*]{\email{jaderebrasil@gmail.com}}
\affil[$\dagger$]{\email{elismar.oliveira@ufrgs.br}}
\affil[$\ddagger$]{\email{rafars@mat.ufrgs.br}}
\date{\vspace{-20px}\small\today}

\begin{document}
    \makeatletter
    \def\blfootnote{\gdef\@thefnmark{}\@footnotetext}
    \let\@fnsymbol\@roman
    \makeatother

\maketitle

\begin{abstract}
    This paper introduces a theory of Thermodynamic Formalism
for Iterated Function Systems with Measures (IFSm). We study the spectral
properties of the Transfer and Markov operators associated to a IFSm.
We introduce variational formulations for the topological entropy
of holonomic measures and the topological pressure of IFSm given by a potential.
A definition of equilibrium state is then natural and we prove
its existence for any continuous potential.
We show, in this setting, a uniqueness result for the equilibrium state
requiring only the G\^ateaux differentiability
of the pressure functional.
\end{abstract}


\section{Introduction}
The modern study of Iterated Function Systems (IFS for short) come back to the early 80's with the works of J. Hutchinson \cite{hutchinson1981fractals} and M. Barnsley \cite{BarnDemko85}. In these papers, the theory was unified both in the geometric and the analytical point of view, generating what we call today the Hutchinson-Barnsley theory for IFS. An IFS is a family of maps acting from a set to itself. Under good contraction hypothesis there exist an invariant compact set  called the fractal attractor. Moreover, if we add weights having good continuity hypothesis, the IFS acts on probabilities having an invariant probability whose support is the fractal attractor set. We observe that several works on geometric features of fractals were done in the previous decades by Mandelbrot and others, but after the 80's the IFS assumed the central role in the generation and study of fractals and its applications.

For a dynamical system given by a map $T: X \to X$, an initial point $x_0 \in X$ is iterated by $T$ producing the orbit $\{x_0, T(x_0), T^2(x_0), ... \}$, whose limit or the cluster points are the objects of main interest, from a dynamical point of view. On the other hand, for an IFS $(X, \tau_\theta)_{\theta \in \Theta}$, we iterate the initial point by choosing at each step a possibly different map $\tau_{\theta}: X\to X$, indexed by the generally finite set $\Theta$, producing multiple orbits $\{Z_j, j \geq 0\}=\{Z_0=x_0, Z_1=\tau_{\theta_0}(x_0), Z_2= \tau_{\theta_1}(\tau_{\theta_0}(x_0)), ... \}$. We notice that the orbit is now a set of orbits controlled by the sequence $(\theta_0, \theta_1, ...) \in \Theta^{\mathbb{N}}$. To avoid this complication Hutchinson defined the fractal operator $F: K(X) \to K(X)$ by
$$F(B)=\bigcup_{\theta \in \Theta}\tau_{\theta}(B),$$
for $B \in K(X)$, where $K(X)$ is the family of nonempty compact sets of $X$. This operator is called the Hutchinson-Barnsley operator and a compact set is invariant or fractal if $F(\Omega)=\Omega$. Additionally $\Omega$ is a fractal attractor if the orbit of $B$ by $F$, given by $\{B, F(B), F^2(B), ... \}$ converge, w.r.t. the Hausdorff-Pompeiu metric to $\Omega$, for any $B \in K(X)$ (see \cite{berinde2013role} for details on the Hausdorff-Pompeiu metric).

%\newpage ORIGINAL

Other possible point of view to understand the dynamics of an IFS is the probabilistic one. Suppose
$X$ is now a metric space and let $\mathcal{P}(X)$ be the set of Borel probability measures defined on $X$.
In this case we consider that, in each step the function to be iterated is chosen according to some probability, which produces a stochastic process $\{X_{0}, X_{1}, X_{2}, ...\}$ where each $X_{j+1}\in X$ is randomly chosen according to a distribution which is obtained from the previous $X_{j}$ by a transition kernel using the IFS law. To illustrate that, we consider the classical case of IFS with constant probabilities studied by Hutchinson, Barnsley and many others in the beginnings of the 80's: we consider $\Theta=\{1,2,..., n\}$, meaning that we have a finite  number of maps, and each one is chosen according to a probability $p_j>0$ where $p_1+ \cdots + p_n=1$, constituting an IFS with probabilities (IFSp for short). We denote by $C(X,\R)$ the Banach space of all real continuous functions equipped with supremum norm
$\| \; \|_{\infty}$. Under this conditions the classic transfer operator (also called Ruelle operator, see Ruelle~\cite{ruelle1967variational,ruelle1968statistical}, Walters~\cite{walters1975ruelle} and Fan~\cite{fan1999iterated}) is given by
\begin{equation}\label{eq:class_transf_op}
  B_{q}(f)(x)=\int_{\Theta} f(\tau_{\theta}(x)) dq_{x}(\theta)= \sum_{j=1}^{n} p_j \, f(\tau_{j}(x)),
\end{equation}
for any $f \in C(X,\R)$, where the measure $q_{x}$ is given by $dq_{x}(\theta)= \sum_{j=1}^{n} p_j  \delta_{j}(\theta), \; \forall x \in X$.
The dual of $B_{q}$ is a Markov operator (see \cite{Doob48}, \cite{Lasota2002} and\cite{BarnDemko85}, for details on Markov operators and its connection with IFS) acting on $\mu$. The operator $\mathcal{L}_{q}(\mu):=B_{q}^*(\mu)$, is implicitly defined by the property
\begin{equation}\label{eq:dual IFS}
   \int_{X} f d\mathcal{L}_{q}(\mu)=  \int_{X} B_{q}(f)(x) d\mu,
\end{equation}
for any $f \in C(X,\R)$.
Given an initial distribution $\mu_{0}$, we iterate it by the Markov operator $M: \mathcal{P}(X) \to \mathcal{P}(X)$ obtaining the distributions  $\mu_{0}, \mu_{1}=\mathcal{L}_{q}(\mu_{0}), \mu_{2}=\mathcal{L}_{q}^2(\mu_{0}), ... \in \mathcal{P}(X)$. We have $X_n \sim \mu_n$ for all $n \geq 0$.  Analogously to the fractal attractor, we say that $\mu \in \mathcal{P}(X)$ is an invariant measure if $\mathcal{L}_{q}(\mu)=\mu$ and that $\mu \in \mathcal{P}(X)$ is an attracting invariant measure (or Hutchinson-Barnsley measure) if $\mathcal{L}_{q}^j(\mu_{0})$ converge to $\mu$ w.r.t. the Monge-Kantorovich metric(see \cite{hutchinson1981fractals}), for any $\mu_{0}\in \mathcal{P}(X)$. It is possible to prove that the support of the invariant attracting measure is the fractal attractor (see \cite{hutchinson1981fractals}).

The final feature of an IFS dynamics we need to understand is the connections between IFS orbits and  invariant measures. The first connection is given by a celebrated result due to M. Barnsley, known as the Chaos Game Theorem (CGT for short), which claims that, from the initial probabilities $p_j$'s, we can build a probability $\mathbb{P}$ over the space $\Theta^{\mathbb{N}}$ such that for $\mathbb{P}$-a.e. $(\theta_0, \theta_1, ...) \subset \Theta^{\mathbb{N}}$ the corresponding orbit $\{x_0, \tau_{\theta_0}(x_0), \tau_{\theta_1}(\tau_{\theta_0}(x_0)), ... \}$ approximate the fractal attractor $\Omega$, for any initial point $x_0$. The second connection is given by the Elton's Ergodic Theorem (EET for short) \cite{elton1987ergodic} claiming that, from the initial probabilities $p_j$'s, we can build a probability $\mathbb{P}$ over the space $\Theta^{\mathbb{N}}$ such that $\mathbb{P}$-a.e. $(\theta_0, \theta_1, ...) \subset \Theta^{\mathbb{N}}$, the corresponding asymptotic average of visits of the orbit $\{x_0, \tau_{\theta_0}(x_0), \tau_{\theta_1}(\tau_{\theta_0}(x_0)), ... \}$ to a measurable set $B \subset X$ is equal to $\mu(B)$, if $\mu(\partial(B))=0$, analogously to the usual Birkhoff ergodic theorem for a single map, where $\mu$ is the invariant measure of the IFS in consideration. For continuous functions it means that
$$\frac{1}{N} \left(f(x_0)+f(\tau_{\theta_0}(x_0))+\cdots + f(\tau_{\theta_{N-1}}( \cdots\tau_{\theta_0}(x_0)))\right) \to \int_X f d\mu,$$
for any $f \in C(X,\R)$, as $N \to \infty$. In other words
$$\frac{1}{N} \left(\delta_{x_0}+\delta_{\tau_{\theta_0}(x_0)}+\cdots + \delta_{\tau_{\theta_{N-1}}( \cdots\tau_{\theta_0}(x_0))}\right) \to \mu,$$
in distribution. In synthesis, the CGT and the EET are random procedures to approximate the fractal attractor and the invariant measure, respectively.

The study of the conditions under which we have, for a given IFS, a fractal
attractor which is the support of an invariant measure is called the
Hutchinson-Barnsley theory. Such conditions have been extremely relaxed and generalized in several ways in the past forty years. A first generalization, yet for $\Theta=\{1,2,..., n\}$, was for IFSp where the constant probabilities $p_j>0$ such that $p_1+ \cdots + p_n=1$ were replaced by variable probabilities $p_j(x)>0$ where $p_1(x)+ \cdots + p_n(x)=1$ for all $x \in X$. Now, the transfer operator is defined by
\begin{equation}\label{eq:class_variable_transf_op}
  B_{q}(f)(x)=\int_{\Theta} f(\tau_{\theta}(x)) dq_{x}(\theta)= \sum_{j=1}^{n} p_j(x) \, f(\tau_{j}(x)),
\end{equation}
where the measure $q_{x}$ is given by $dq_{x}(\theta)= \sum_{j=1}^{n} p_j(x)  \delta_{j}(\theta), \; \forall x \in X$,  for any $f \in C(X,\R)$, Very general conditions for the existence of the invariant measure for such IFS are given in \cite{barnsley1988invariant}. We point out that the EET was also proved for variable probabilities and finite functions in \cite{elton1987ergodic}.

In Fan~\cite{fan1999iterated}, 1999, the condition $p_1(x)+ \cdots + p_n(x)=1$ is finally dropped assuming only that each $p_\theta(x)\geq 0$ for $\theta \in \{1,...,n\}$. In this work Fan study a contractive system which is a triplet  $(X, \tau_\theta, p_\theta )_{\theta \in \{1,...,n\}}$, where each $\tau_\theta$ is a contractive map and each $p_\theta(x)\geq 0$ for $\theta=1,...,n$, generalizing the notion of IFS with probabilities. In this setting, Fan proves a Ruelle-Perron-Frobenius theorem (RPF theorem, for short), meaning the existence of a positive eigenfunction for the operator $B_{q}$ and an eigenmeasure for the dual operator $B_{q}^*$ with the same eigenvalue which is the spectral radius of $B_{q}$.

The next key improvement was given by Stenflo~\cite{stenflo2002uniqueness}, where random iterations are used to represent the iterations of a so called IFS with probabilities, $(X, \tau_\theta, \mu)_{\theta \in \Theta}$ for an arbitrary measurable space $\Theta$. The approach here is slightly different from the previous works on IFS with probabilities, instead considering weights, the iterations from $Z_0 \in X$ are $Z_{j+1}=\tau_{I_j}(Z_{j})$ governed by a sequence of i.i.d variables $\{I_j \in \Theta\}_{j \in \mathbb{N}}$, with distribution $\mu$, generating a Markov chain $\{Z_{j},\; j \geq 0\}$ with transfer operator given by
\begin{equation}\label{eq:stenflo_transf_op}
  B_{q}(f)(x)=\int_{\Theta} f(\tau_{\theta}(x)) dq_{x}(\theta)= \sum_{j=1}^{n} p_j(x) \, f(\tau_{j}(x)),
\end{equation}
where the measure $q_{x}$ is given by $dq_{x}(\theta)= d\mu(\theta), \; \forall x \in X$,  the for any $f \in C(X,\R)$, The main goal of Stenflo~\cite{stenflo2002uniqueness} is to establish, when $B_{q}$ is Feller, the existence of  an unique attracting invariant measure $\pi$, for this Markov chain.

The approach in our paper is a generalization of Stenflo~\cite{stenflo2002uniqueness}, \eqref{eq:stenflo_transf_op}, and all classical settings \eqref{eq:class_transf_op} and \eqref{eq:class_variable_transf_op}. We take a family $q_{x}(\cdot) \in \mathcal{M}(\Theta)$,
indexed by  $x \in X$, generating a Markov chain with transfer operator given by
$$B_q(f)(x)=\int_{\Theta} f(\tau_\theta(x)) dq_{x}(\theta),$$
for any $f \in C(X,\R)$. The meaning of the distribution $q_{x}(\cdot) \in \mathcal{M}(\Theta)$ is such that, the position $x$ of a previous iteration of the IFS determine the distribution $q_{x}(\cdot)$ of $\theta$ used to choose the function $\tau_\theta$ and produce the new point $\tau_\theta(x)$. When $q_{x}=\mu$, for any $x \in X$, is a constant distribution we recover the setting from Stenflo~\cite{stenflo2002uniqueness}.

In our setting
the IFSm $(X, \tau_\theta, q)_{\theta \in \Theta}$ can be studied as the sample paths of the Markov process $\{Z_{j},\; j \geq 0\}$ with initial distribution $\mu_0=\mu \in \mathcal{M}(X)$ and  $\mu_{j+1}= \mcal{L}_{  q}(\mu_{j})$, where for any $\nu \in \mathcal{M}(X)$,
$$\int_{X} f(x) d\mcal{L}_{  q}(\nu)(x)=  \int_{X} B_q(f)(x) d\nu(x),$$
for any $f \in C(X,\R)$. Such degree of generality is necessary to enlarge the range of application for the IFS theory, specially the thermodynamic formalism. In Section \ref{application_economics} we present a situation where we believe the tools developed in the previous sections can be applied when analyzing
an interesting problem in economics.

Our goal is to present a complete theory of thermodynamical formalism for these IFS with measures, that is, good definitions for transfer operators, invariant measures, entropy, pressure, equilibrium measures and a variational principle. Finally, we want to use these tools to characterize the solutions of the ergodic optimization problem.

For sake of completeness we would like to point out that we do not prove a  RPF theorem for those systems, only the existence of positive eigenfunctions, but we establish all the results that can be derived if we have assumed such a property. To the best of our knowledge the RPF theorem for IFSm has not been established an it is a very hard problem. There are several works on the matter of finding IFS for which the RPF theorem holds.  Those IFS are said to have the RPF property.  In 2009 Lopes and Oliveira~\cite{lopes2009entropy} studied those systems renaming it as weighted systems or IFS with weights, having the RPF property, producing a self contained notion of  entropy and topological pressure through a variational principle for holonomic measures allowing to establish a thermodynamical formalism for IFS. Other approaches for IFS thermodynamic formalism were developed by Urba\'{n}ski~\cite{simon2001invariant,mauldin2000parabolic,hanus2002thermodynamic}  and many others.

It's worth to mention that, in  Urba\'nski et al.~\cite{hanus2002thermodynamic} a thermodynamic formalism for conformal infinite (countable) iterated function systems is presented using the conformal structure via partition functions.
Also in K\"aenm\"aki~\cite{kaenmaki2004natural} a thermodynamical formalism for IFS is studied with the help of cylinder functions, where general IFS means that $(X, \tau_\theta)_{\theta \in \Theta}$ and $\Theta$ is the increasing union of finite alphabets. In Lopes et al.~\cite{lopes2015entropy} a thermodynamic formalism for shift spaces, taking values on a compact metric space is presented, although this problem is closely related to thermodynamic formalism for IFS when we associate the pre images of the shift map with a respective maps producing an infinite IFS. Also in \cite{aguiar2018variational} a variational principle for the specific entropy on the context of
symbolic dynamics of compact metric space alphabets was developed generalizing somehow the results in \cite{lopes2009entropy}.

In our work we will extend the variational results in Lopes and Oliveira~\cite{lopes2009entropy} and, more recently, the preprint Cioletti and Oliveira~\cite{CiolOLivArxiv2017},  to a general IFS  called IFS with measures (IFSm), $(X, \tau_\theta, q)_{\theta \in \Theta}$ for an arbitrary compact space $\Theta$ (see Dumitru~\cite{dumitru2013attractors} for the Hutchinson-Barnsley theory for such infinite systems or Lukawska~\cite{gwozdz2005hutchinson} for infinite countable ones. Other related articles are \cite{MR1302367, MR2984148, MR2506506, MR1418125, MR1387085, MR3272642}). We point out that we do not use partitions to construct our entropy, only an integral formulation, and the variational principle is taken over the holonomic measures which enclose the invariant ones.

The structure of the paper is the following: in Section \ref{basics}, we present the basic definitions on IFS with measures (IFSm) and a fundamental result about the eigenspace associated to the maximal eigenvalue of the transfer operator. We also prove the existence of a positive eigenfunction for the transfer operator associated to the spectral radius and give a constructive proof of the existence of equilibrium states.
In Section \ref{Markov_Operator} we define the Markov operator, which in the case of a normalized IFSm gives the evolution of the distribution of the associated Markov process, and show that the set of eigenmeasures for it is non-empty.
In Section \ref{sec-holonomic}, we introduce holonomic measures, which play the role of invariant measures in the IFS setting.
In Section \ref{entropy-and-pressure} we define entropy for a IFSm, the topological pressure of a given potential function, as well as the concept of equilibrium states.
In a remark in the end of this section we show how the classical thermodynamical formalism for a dynamical system is a particular case of the  IFSm Thermodynamic Formalism.
In Section \ref{uniqueness-for-eqstates}
a uniqueness result for the equilibrium states  is obtained.
Finally, in Section \ref{application_economics} we present a possible application in economic theory of the theory developed in the previous sections.

\section{IFSm and Transfer Operator}\label{basics}
In this section we set up the basic notation and
present a fundamental result about the eigenspace associated to
the maximal eigenvalue (or spectral radius) of the transfer operator.

\bigskip

In this paper $\X,\Theta$ are compact metric spaces, equipped with $\mscr{B}(\X)$ and $\mscr{B}(\Theta)$ respectively the Borel $\sigma$-algebra for $\X$ and $\Theta$.

The Banach space of all real continuous functions equipped with supremum norm is denoted by $C(\X,\R)$.
Its topological dual, as usual, is identified with $\mscr{M}_{s}(\X)$,
the space of all finite Borel signed measures endowed with
total variation norm. We use the notation $\mscr{M}_{1}(X)$
for the set of all Borel probability measures over $X$
supplied with the weak-$*$ topology. Since we are assuming
that $X$ is a compact metric space then we have that the topological
space $\mscr{M}_{1}(X)$ is compact and metrizable.

Take $\q={(\q_{x})}_{x\in\X}$ a collection of measures on $\mscr{B}(\Theta)$, such that
\begin{enumerate}[label= (\textit{$\q$\arabic*})]
  \item $\ds\supqx\equiv\sup_{x\in X}\q_{x}(\Theta) < \infty$,
  \item $\ds\infqx\equiv\inf_{x\in\X}\q_{x}(\Theta) > 0$,
  \item\label{q-measurable} $\ds x \mapsto \q_{x}(A)$ is a Borel map, i.e, is $\mscr{B}(\X)$-measurable for all fixed $A\in\mscr{B}(\Theta)$,
  \item\label{q theta continuous} $\ds x \mapsto \q_{x}$ is weak$^{*}$-continuous.
\end{enumerate}

An Iterated Function System with measures $q$, IFSm for short, is a triple
$\Rq=(\X, \tau, \q)$, where $\tau={(\tau_{\theta})}_{\theta\in\Theta}$ is a collection of functions from $\X$ to itself with the following property
\begin{enumerate}[label= (\textit{$\tau$\arabic*})]
  \item\label{tau continuous} $\tau: (\Theta, \X) \mapsto \X$, where $\tau(\theta, x) = \tau_{\theta}(x)$ is continuous.
\end{enumerate}

The $\Rq$ is said to be {\bf normalized} if for all $x\in\X$, $\q_{x}$ is a probability measure.

\begin{defi} Let $\Rq=(\X,\tau,\q)$ be an IFSm.
The \textbf{Transfer Operator} $\Bq:\toitself{C(\X,\R)}$ associated to
$\mcal{R}_{\q}$ is defined by:
  \[
    \Bq(f)(x) = \int_{\Theta}f(\tau_{\theta}(x))\dq_{x}(\theta),
    \qquad \forall x \in X.
  \]
\end{defi}
$\Bq$ is well defined. In fact, $\Bq$ is continuous once that
\[
  \norm{\Bq(f)}_{\infty}
  = \sup_{x}\abs{\int\,f(\tau_{\theta}(x))\dq_{x}(\theta)}
  \le \supqx\norm{f}_{\infty}
  < \infty.
\]
Futhermore, for a fixed $f \in C(\X,\R)$ and $x\in\X$, given $\varepsilon > 0$, take $\delta>0$ s.t.
\[
\sup_{\theta\in\Theta}d(f(\tau_{\theta}(x)),f(\tau_{\theta}(y)))<\frac{\varepsilon}{2\supqx},
\]
and
\[
  \abs{\int_{\Theta}f(\tau_{\theta}(x))d\q_{x}(\theta) - \int_{\Theta}f(\tau_{\theta}(x))d\q_{y}(\theta)} < \frac{\varepsilon}{2},
\]
for all $y\in\X$ with $d(x,y) < \delta$. Then,

\begin{align*}
  \vert&\Bq(f)(x) - \Bq(f)(y)\vert
  = \abs{\int_{\Theta\hspace{-6px}}f(\tau_{\theta}(x))\dq_{x}(\theta) - \int_{\Theta\hspace{-6px}}f(\tau_{\theta}(y))\dq_{y}(\theta)}\\
  &\le \int_{\Theta\hspace{-6px}}\abs{f(\tau_{\theta}(x))-f(\tau_{\theta}(y))}\dq_{x}(\theta)
  +\abs{\int_{\Theta\hspace{-6px}}f(\tau_{\theta}(y))\dq_{x}(\theta)-
  \int_{\Theta\hspace{-6px}}f(\tau_{\theta}(y))\dq_{y}(\theta)}\\
  &<\frac{\varepsilon}{2\supqx}\int_{\Theta}\dq_{x}(\theta)
  + \frac{\varepsilon}{2}
  = \frac{\varepsilon}{2} + \frac{\varepsilon}{2}=\varepsilon.
\end{align*}

This shows that, for $f\in C(\X,\R)$ and $x\in\X$, given $\varepsilon>0$, there is $\delta>0$ such that for every
$d(x,y)<\delta$, $\abs{\Bq(f)(x) - \Bq(f)(y)} < \epsilon$, therefore $\Bq(f)(x)$ is continuous.
\bigskip

When dealing with the spectral radius of a positive operator, via Gelfand's formula, we will need to consider the norm of iterates of such operator, which is defined by a supremum which is attained in the constant function, as we will see in the proof of Theorem~\ref{powers-description}. For this reason we will need the following proposition:

\begin{prop}\label{Bq-powers1}
	Let $\Rq=(\X, \tau, \q)$ be a IFSm.
	Then for the $N$-th iteration of $\Bq$ we have
	\[
	\Bq^{N}(1)(x) =
	\int_{\Theta^{N}}
	\dP(\theta_{0},\ldots, \theta_{N-1})
	\]
	where,
  \vspace{5px}

	$\ds\dP(\theta_{0},\ldots, \theta_{N-1})
	\equiv \prod_{j=1}^{N}\dq_{x_{N-j}}(\theta_{N-j})$,
	$x_0=x$ and $x_{j+1}=\tau_{\theta_{j}}x_j$.
\end{prop}

\bigskip

\begin{proof}
  This expression can be obtained by proceeding a formal induction on $N$.
  For $N=2$ and $x=x_{0}$, we have

\begin{align*}
  \Bq^{2}(1)(x_{0})&=\int_{\Theta} \Bq(1)(\tau_{\theta_{0}}(x_{0}))\dq_{x_{0}}(\theta_{0})\\
  &=\int_{\Theta}\int_{\Theta}\dq_{x_{1}}(\theta_{1})\dq_{x_{0}}(\theta_{0})\\
  &=\int_{\Theta^{2}}\dP(\theta_{0},\theta_{1}).
\end{align*}

And, if
\[\Bq^{N}(1)(x) = \int_{\Theta^{N}}\dP(\theta_{0},\ldots,\theta_{N-1}),\]

then
\begin{align*}
  \Bq^{N+1}(1)(x)&= \int_{\Theta}\Bq^{N}(1)(x_{1})\dq(\theta_{0})\\
  &=\int_{\Theta}\int_{\Theta^{N}}\dPx{x_{1}}(\theta_{1},\ldots,\theta_{N})\dq(\theta_{0})\\
&=\int_{\Theta}\cdots\int_{\Theta}\left(\prod_{j=0}^{N-1}\dq_{x_{N-j}}(\theta_{N-j})\right)\dq(\theta_{0})\\
&=\int_{\Theta}\cdots\int_{\Theta}\left(\prod_{j=1}^{N+1}\dq_{x_{N+1-j}}(\theta_{N+1-j})\right)\\
&=\int_{\Theta^{N}}\dP(\theta_{0},\ldots\,\theta_{N}).
\end{align*}
\end{proof}

\begin{remark}
  The formal notation used for $\Px$, in fact, means that $\Px$ is a measure in $\Theta^{N}$ defined by,

  \[\Px(\Theta_{0}\times\cdots\times\Theta_{N-1}) = \int_{\Theta_{0}}\cdots\int_{\Theta_{N-1}}\dq_{x_{N-1}}(\theta_{N-1})\cdots\dq_{x_{0}}(\theta_{0}).\]

  In the case $N=2$ for instance,
  \begin{align*}
    \Px(\Theta_{0}\times\Theta_{1}) &=\int_{\Theta_{0}}\int_{\Theta_{1}}\dq_{\tau_{\theta_{0}}x}(\theta_{1})\dq_{x}(\theta_{0})\\
    &= \int_{\Theta_{0}}\q_{\tau_{\theta_{0}}x}(\Theta_{1})\dq_{x}(\theta_{0}).
  \end{align*}

  Note that $\q_{\tau_{\theta_{0}}(x_{0})}(\Theta_{1})$, with fixed $\Theta_{1}$ and $x_{0}$, is a function of $\theta_{0}$ that is measurable: indeed, if $A\in\mscr{B}(\Theta)$, $f_{A}:\X\to\R$ defined by $f_{A}(x) = q_{x}(A)$ is measurable by $\ref{q-measurable}$ and by $\ref{tau continuous}$ implies $\tau$ is measurable. Thus, $F_{A} := f_{A}\circ \tau$ is measurable.

  \begin{prop}\label{prop-measurable-qx}
    If $f:\X\to\R$ is a measurable nonnegative function, then

    \[H(x) := \int_{\Theta_{0}} f\circ\tau(\theta, x)\dq_{x}(\theta),\]

    is \textit{measurable}.
  \end{prop}

  Using Proposition $\ref{prop-measurable-qx}$, it is a simple induction to prove that
  \[
  x\mapsto
   \Px(\Theta_{0}\times\cdots\times\Theta_{N-1})=
                \int_{\Theta_{0}}\Pxx{\tau_{\theta_{0}}x_{0}}(\Theta_{1}\times\cdots\times\Theta_{N-1})\dq_{x_{0}}(\theta_{0})
\]
  is measurable for any $\Theta_{i}\in\mcal{B}(\Theta)$.

  In this way we conclude that $\Px$ is well defined for each space $\Theta^{N}$.
\end{remark}

\begin{theorem}\label{powers-description}
Let $\Rq=(\X,   \tau,   \q)$ be a  IFSm
and suppose that there are a positive number $\rho$ and
a strictly positive continuous function
$h: \X \to \R$ such that  $\Bq(h)=\rho h$.
Then the following limit exits
\begin{align}\label{1overNlnBN}
\lim_{N \to \infty}
\frac{1}{N}
\ln\left(\Bq^{N}(1) (x) \right)
=
\ln \rho
\end{align}
the convergence is uniform in $x$ and
$\rho=\rho(\Bq)$ is the spectral radius of $\Bq$ acting on $C(X,\R)$.
\end{theorem}
\begin{proof}
  From the hypothesis we can build a normalized IFSm $\Rp=(\X,\tau,\p)$ where
  \[ \dpp_{x}(\theta) = \frac{h(\tau_{\theta}(x))}{\rho h(x)}\dq_{x}(\theta).\]

  Note that $\dP$ and $\dP[\p]$ are related in the following way

  \begin{align*}
    \dP(\theta_{0},\ldots,\theta_{N-1})
    &=\prod_{j=1}^{N}\dq_{x_{N-j}}(\theta_{N-j})\\
    &=\prod_{j=1}^{N}\frac{\rho h(x_{N-j})}{h(\tau_{\theta_{N-j}}(x_{N-j}))}\dpp_{x_{N-j}}(\theta_{N-j})\\
    &=\rho^{N}\prod_{j=1}^{N}\frac{h(x_{N-j})}{h(x_{N-j+1})}\dpp_{x_{N-j}}(\theta_{N-j})\\
    &=\rho^{N}\frac{h(x_{0})}{h(x_{N})}\prod_{j=1}^{N}\dpp_{x_{N-j}}(\theta_{N-j})\\
    &=\rho^{N}\frac{h(x_{0})}{h(x_{N})}\dP[\p](\theta_{0},\ldots,\theta_{n-1})
  \end{align*}
where	$x_0=x$ and $x_{j+1}=\tau_{\theta_{j}}x_j$.

Since $\X$ is compact and $h$ is a strictly positive continuous function, we have
   constants $0<a<1<b$, independent of $x$, such that
  \[ a \le h(x_{0})/h(x_{N}) \le b.\]

  Using the Proposition $\ref{Bq-powers1}$ and the above inequalities, we
  obtain for any fixed $N\in\N$ the following expression

  \begin{align*}
    \frac{1}{N}\ln(\Bq^{N}(1)(x))
    &= \frac{1}{N}\ln\left(
        \int_{\Theta^{N}}
        \dP(\theta_{0},\ldots, \theta_{N-1})
      \right)\\
    &= \frac{1}{N}\ln\left(
        \int_{\Theta^{N}}
        \rho^{N}\frac{h(x_{0})}{h(x_{N})}\dP[\p](\theta_{0},\ldots, \theta_{N-1})
      \right)\\
    &= \ln\rho
    + \frac{1}{N}\ln\left(
        \int_{\Theta^{N}}
        \frac{h(x_{0})}{h(x_{N})}\dP[\p](\theta_{0},\ldots, \theta_{N-1})
      \right).
  \end{align*}

  Furthermore,
  \begin{align*}
    \frac{1}{N}\ln\left(
        \int_{\Theta^{N}}
        \frac{h(x_{0})}{h(x_{N})}\dP[\p](\theta_{0},\ldots, \theta_{N-1})
      \right)
    &\ge
    \frac{1}{N}\ln\left(
        \int_{\Theta^{N}}
        a\dP[\p](\theta_{0},\ldots, \theta_{N-1})
    \right)\\
    &= \frac{1}{N}\ln a + \frac{1}{N}\ln\int_{\Theta_{N}}\dP[\p]\\
    &= \frac{1}{N}\ln a \xrightarrow{\,N\to\infty\,} 0\\
  \end{align*}
  and

  \begin{align*}
    \frac{1}{N}\ln\left(
        \int_{\Theta^{N}}
        \frac{h(x_{0})}{h(x_{N})}\dP[\p](\theta_{0},\ldots, \theta_{N-1})
      \right)
    &\le
    \frac{1}{N}\ln\left(
        \int_{\Theta^{N}}
        b\dP[\p](\theta_{0},\ldots, \theta_{N-1})
    \right)\\
    &= \frac{1}{N}\ln b + \frac{1}{N}\ln\int_{\Theta_{N}}\dP[\p]\\
    &= \frac{1}{N}\ln b\xrightarrow{\,N\to\infty\,} 0.
  \end{align*}

  Therefore, for every $N\ge 1$ we have
\[\sup_{x\in\X}\abs{\frac{1}{N}\ln\left(\Bq^{N}(1)(x)\right) - \ln\rho} \leq \frac{\ln b-\ln a }{N},\]
  which proves ($\ref{1overNlnBN}$).
  Now using Gelfand's formula
  for the spectral radius
  and the fact that, for a positive operator $T$ defined on $C(X,\R)$
  we have $\norm{T}=\left\| T(1) \right\|_{\infty}$, where $\norm{T}$ denotes the usual norm operator
  $\norm{T}= \sup_{\{ \|f\|_{\infty} \leq 1\}} \|T(f) \|_{\infty}$,
    we have
  \begin{align*}
  \abs{\ln\rho(\Bq) - \ln\rho}
  &=
  \abs{\ln\left( \limsup_{N\to\infty} \norm{\Bq^{N}}^{\frac{1}{N}}\right) - \ln\rho}
  =
  \limsup_{N\to\infty}
  \abs{\frac{1}{N} \ln\norm{\Bq^{N}} - \ln\rho}
  \\
  &=
  \limsup_{N\to\infty}
  \abs{ \frac{1}{N} \ln\left( \left\|\Bq^{N}(1)\right\|_{\infty} \right) - \ln\rho }
  \\
  &\leq
  \limsup_{N\to\infty}
  \ \ \sup_{x\in\X}
  \abs{\frac{1}{N} \ln\left( \Bq^{N}(1)(x) \right) - \ln\rho}
  \\
  &\leq
  \limsup_{N\to\infty}
  \frac{\ln b -\ln a }{N}
  =0.
\end{align*}

\end{proof}

We will now address the question of existence of positive eigenfunctions  for the transfer operator associated to the spectral radius, and give a constructive proof of the existence of equilibrium states.

Let $\Rq=(\X, \tau, \q)$ assuming that there is $\mu$ a probability on $\Theta$ s.t. $\forall x \in \X$, $\q_{x} \ll \mu$ and $J:\Theta\times\X\to\R$, defined by $J(\theta,x) := \frac{\dq_{x}}{\dmu}(\theta)$, a continuous function. Define $u(x,\theta) = \log J(\theta,x)$, and consider a parametric family of variable discount functions $\delta_{n}:\intervalco{0,+\infty}\to\R$, where $\delta_{n}(t) \to I(t) = t$, when $n\to\infty$, pointwise and the normalized limits $\lim_{n}w_{n}(x) - \max w_{n}$ of the fixed points
\[
w_{n}(x) := \log\int_{\Theta}e^{u(\theta,x)+\delta_{n}(w_{n}(\tau(\theta,x)))}\dmu = \log\int_{\Theta}e^{\delta_{n}(w_{n}(\tau(\theta,x)))}\dq_{x}(\theta)
\]
for a discounted transfer operator (see\cite{cioletti2019}, Definition 3.2 for the operator definition and Theorem 3.6 for the existence of the fixed points) associated to a variable discount decision-making process. We now check that the requirements in $\cite{cioletti2019}$ are fulfilled in our setting. Consider $S_{n}:=(\X,\Theta,\Psi,\tau,u,\delta_{n})$ where $\Psi(x) = \Theta$ for all $x\in\X$ and the sequence $(\delta_{n})$ satisfies the admissibility conditions:

\begin{enumerate}
  \item the contraction modulus $\gamma_{n}$ of $\delta_{n}$ is also a variable discount function;
  \item $\delta_n(0) = 0$ and $\delta_{n}(t) \le t$ for any $t\in(0,+\infty)$;
  \item for any fixed $\alpha > 0$ we have $\delta_{n}(t+\alpha) - \delta_{n}(t) \to \alpha$ when $n\to\infty$, uniformly in $t > 0$.
\end{enumerate}

\begin{theorem}
  Let $\Rq$ and $(\delta_{n})$ in above conditions such that the above defined $u$ satisfies:
  \begin{enumerate}
    \item $u$ is uniformly $\delta$-bounded for $(\delta_{n})$;
    \item $u$ is uniformly $\delta$-dominated for $(\delta_{n})$.
  \end{enumerate}
  Then there exists a positive and continuous eigenfunction $h$ such that $\Bq(h) = \rho(\Bq)h$.
\end{theorem}

\begin{proof}
  Theorem 3.28 of $\cite{cioletti2019}$ implies that there exists $k\in[0,\norm{u}_{\infty}]$ and $\varphi(x):=e^{h(x)}$ continuous and positive function with

 \[e^{k}\varphi(x) = \int_{\Theta}\varphi\circ\tau(\theta,x)\,e^{u(x,\theta)}\dmu(\theta) = \Bq(\varphi)(x),\]

 for all $x\in\X$. Now use the Theorem \ref{powers-description} and the theorem is proven.
\end{proof}

\section{Markov Operator and its Eigenmeasures}\label{Markov_Operator}

In this section we define the Markov Operator, which in the case of a normalized IFSm gives the evolution of the distribution of the associated Markov Process, and show that the set of eigenmeasures for it is non-empty.

\smallskip

\begin{defi}
	The \textbf{Markov Operator}
	$\Lq : \toitself{\mscr{M}_{s}(X)}$
	is the unique bounded linear operator satisfying
	\[
	\int_{\X} f\, \text{d}[\mcal{L}_{  q} (\mu)] =
	\int_{\X } \Bq(f)\, \text{d}\mu,
	\]
	for all $\mu \in \mscr{M}_s(X)$ and $f\in C(\X,\R)$.
\end{defi}

In the case of a normalized IFSm, we can consider the Markov Process $\{Z_j, j\geq 0\}$ with initial distribution $Z_0\sim \mu_0$, where $\mu_0 \in \mscr{M}_{1}(X)$, and
$Z_{j+1}=\tau_{{\theta}_j}(Z_j)$ for $j \geq 0$, where ${{\theta}_j} \sim q_{Z_j}$. Then, if $Z_j \sim \mu_j$, we have $\mu_{j+1}= \Lq(\mu_j)$.

\begin{theorem}\label{prop-exist-auto-medida}
Let $\Rq=(\X,   \tau,  \q)$ be a IFSm.
Then there exists a positive number $\rho\leq\rho(\Bq)$ such that the set
$
\mcal{G}^{*}(\q)
=
\{
  \nu \in \mathscr{M}_1(X): \mathcal{L}_{  q}\nu =\rho\nu
\}
$
is not empty.
\end{theorem}
\begin{proof}
Notice that the mapping
\[
\mscr{M}_1(X)\ni
\gamma
\mapsto
\frac{\mathcal{L}_{  q}(\gamma) } {\mathcal{L}_{  q}(\gamma)(X)}
\]
sends $\mathscr{M}_1(X)$ to itself. From its convexity
and compactness, in the weak topology which is Hausdorff when $X$ is metric and compact,
it follows from the continuity of $\mathcal{L}_{  q}$ and the Tychonov-Schauder Theorem
that there is at least one probability measure $\nu$ satisfying
$\mathcal{L}_{  q}(\nu)=(\mathcal{L}_{  q}(\nu)(X))\, \nu$.

We claim that
\begin{align}\label{des1-exi-auto-medida}
\inf_{x\in\X}\q_{x}(\Theta)
\leq
\mathcal{L}_{  q}(\gamma)(\X)
\leq
\sup_{x\in\X} \q_{x}(\Theta)
\end{align}
for every $\gamma\in \mathscr{M}_1(X)$.

Indeed,
  \[\Bq(1)(x) = \int_{\Theta}1\dq_{x}(\theta) = \q_{x}(\Theta),\]
  \[\Lq(\gamma)(\X) = \int_{\X}1\,\text{d}[\Lq\gamma]=\int_{\X}\Bq(1)\dgamma = \int_{\X}\q_{x}(\Theta)\dgamma(x),\]
  \[0 < \inf_{x\in\X}\q_{x}(\Theta) \le \int_{\X}\q_{x}(\Theta)\dgamma(x) \le \sup_{x\in\X}\q_{x}(\Theta) < \infty.\]

From the inequality $\eqref{des1-exi-auto-medida}$ it follows that
\[
0
<
\rho
\equiv
\sup\{ \Lq(\nu)(X):
    \Lq(\nu)
		=
		(\Lq(\nu)(X))\, \nu
	\}
<
+\infty.
\]
By a compactness argument one can show the existence of
$\nu\in \mscr{M}_{1}(\X)$ so that
$\Lq\nu=\rho\nu$.
Indeed, let ${(\nu_n)}_{n\in\mathbb{N}}$
be a sequence such that
$\Lq(\nu_n)(\X)\uparrow \rho$,
when $n$ goes to infinity.
Since $\mscr{M}_1(X)$ is compact metric space
in the weak topology we can assume, up to subsequence,
that $\nu_n\rightharpoonup \nu$. This convergence
together with the continuity of $\Lq$ provides
\[
\Lq\nu
=
\lim_{n\to\infty}\Lq\nu_n
=
\lim_{n\to\infty}\Lq(\nu_n)(X)\nu_n
=
\rho\, \nu,
\]
thus showing that the set
$
\mcal{G}^{*}( {q})
\equiv\{
    \nu \in \mscr{M}_1(X):
      \Lq\nu =\rho\, \nu
\}
\neq
\emptyset
$.

To finish the proof we observe that by using  any
$\nu\in \mcal{G}^{*}(\q)$, we get the following inequality
\begin{align*}
\rho^N
=
\int_{X} \Bq^{N}(1)\, \text{d}\nu
\leq
\norm{\Bq^{N}}.
\end{align*}
From this inequality and Gelfand's Formula, it
follows that $\rho\leq \rho(B_{ {q}})$.
\end{proof}

\section{Holonomic Measure and Disintegrations}\label{sec-holonomic}

Now  we introduce holonomic measures, which play the role of invariant measures in the IFS setting.

An invariant measure for a classical dynamical system $T: \toitself{\X}$ on a compact space
is a measure $\mu$ satisfying  for all $f\in C(X,\R)$
\[
\int_{\X} f(T(x))\,\text{d}\mu=
\int_{\X} f(x)\,\text{d}\mu,
\quad \text{equivalently}\quad
\int_{\X} f(T(x))-f(x)\,\text{d}\mu= 0.
\]
From the Ergodic Theory point of view the natural generalization
of this concept for an IFS $\mcal{R}=(\X, \tau)$ is the
concept of holonomy.

Consider the cartesian product space $\Omega\equiv\X\times\Theta$
and for each $f\in C(X,\R)$ its ``$\Theta$-differential''
$\df: \Omega\to\R$ defined by $[\df[x]](\theta)\equiv f(\tau_{\theta} (x)) -f(x)$.

\begin{defi}\label{invariance}
A measure $\hat{\mu}$ over $\Omega$ is said holonomic,
with respect to an IFS $\mcal{R}$ if
for all $f\in C(X,\R)$ we have
\begin{align*}
\int_{\Omega}[\df[x]](\theta) \, d\hat{\mu}(x,\theta)=0.
\end{align*}
Notation,
\[
\ds \HR\equiv
    \{\hat{\mu}  \, | \, \hat{\mu}
    \text{ is a holonomic probability measure with respect to}\ \mcal{R}\}.
\]
\end{defi}

Since $\Omega$ is compact the set of all holonomic probability measures is obviously convex and compact. It is also not empty because $\Omega$
is compact and any average
\[
\hat{\mu}_{N}
    \equiv \frac{1}{N} \sum_{j=0}^{N-1} \delta_{(x_j, {\theta}_j)},
\]
where $x_{j+1} = \tau_{{\theta}_j}(x_j)$ and $x_0\in \X$ is fixed,
will have their cluster
points in  $\HR$. Indeed, for all $N\geq 1$ we have
the following identity
\begin{align*}
\int_{\Omega} [\df[x]](\theta) \, d\hat{\mu}_{N}(x,\theta)
&= \frac{1}{N} \sum_{j=0}^{N-1} [\df[x_{j}]](\theta_j)
= \frac{1}{N} (f(\tau_{\theta_{N-1}}(x_{N-1}))-f(x_0) ).
\end{align*}

From the above expression is easy to see that
if $\hat{\mu}$ is a cluster point of the sequence ${(\hat{\mu}_N)}_{N\geq 1}$,
then there is a subsequence ${(N_k)}_{k\to\infty}$ such that
\begin{align*}
\int_{\Omega} [\df[x]](\theta) \, d\hat{\mu}(x,\theta)
&= \lim_{k\to\infty}
   \int_{\Omega} [\df[x]](\theta) \, d\hat{\mu}_{N_k}(x,\theta)\\
&= \lim_{k\to\infty}\frac{1}{N_{k}} (f(x_{N_k})-f(x_0) )
=0.
\end{align*}

\begin{theorem}[Disintegration]\label{disintegration thm}
Let $X$ and $Y$ be compact metric spaces,
$\hat{\mu}:\mscr{B}(Y)\to [0,1]$ a Borel probability measure, $T:Y \to X$ a
Borel mensurable function and for each $A\in\mscr{B}(X)$ define a
probability measure $\mu(A)\equiv \hat{\mu}(T^{-1}(A))$. Then there exists a
family of Borel probability measures ${(\mu_{x})}_{x \in X}$ on $Y$,
uniquely determined $\mu$-a.e, such that
\begin{enumerate}
\item $\mu_{x}(Y\backslash T^{-1}(x)) = 0 $, $\mu$-a.e;
\item $\ds
      \int_{Y} f\, d\hat{\mu}
      = \int_{X}\left[\int_{T^{-1}(x)} \!\!\! f(y)\  d\mu_{x}(y)\right] d\mu(x)$.
\end{enumerate}
This decomposition is called the \textit{disintegration}
of $\hat{\mu}$, with respect to $T$.
\end{theorem}
\begin{proof}
For a proof of this theorem, see $\cite{dellacherie1978probabilities}$ p.78 or $\cite{ambrosio2005gradient}$, Theorem 5.3.1.
\end{proof}

In this paper we are interested in disintegrations in cases where
$Y$ is the Cartesian product $\Omega\equiv \X\times\Theta$ and $T:\Omega\to\X$
is the projection on the first coordinate.
In such cases if $\hat{\mu}$ is any Borel probability measure on $\Omega$,
then follows from the first conclusion
of Theorem $\ref{disintegration thm}$ that the disintegration
of $\hat{\mu}$ provides for each $x\in \X$ a unique probability measure
$\mu_{x}$ ($\mu$-a.e.) supported on $\{x\}\times\Theta$.
So we can write the disintegration of $\hat{\mu}$
as $\text{d}\hat{\mu}(x,\theta)= \text{d}\mu_{x}(\theta)\text{d}\mu(x)$, where here we are abusing notation
identifying $\mu_x(\{x\}\times A)$ with $\mu_x(A)$.

Consider any IFSm $\Rq=(\X,\tau,\q)$ and a probability $\nu \in \mathscr{M}_{1}(\X)$ such that the Markov operator associated to the IFSm $\Rq$ satisfies
\[
  \Lq(\nu) = \nu,
\]
then it is possible to define a
holonomic probability measure $\hat{\mu}\in \HR$ given by
$\text{d}\hat{\mu}(x,\theta)= \text{d}q_{x}(\theta) \,\text{d}\nu(x)$.
Indeed, fixed $f \in C(X,\mathbb{R})$, we obtain
\[
  \int_{\Omega}[\df[x]](\theta) \,\text{d}\hat{\mu}(x,\theta) = 0
  \Longleftrightarrow
  \int_{\Omega}f(\tau_{\theta}(x))  \,\text{d}\hat{\mu}(x,\theta)  = \int_{\Omega} f(x) \,\text{d}\hat{\mu}(x,\theta).
\]
But,
\[
  \int_{\X}\int_{\Theta} f(\tau_{\theta}(x))  \,\text{d}q_{x}(\theta)\text{d}\nu(x)
  = \int_{\X}\int_{\Theta} f(x) \,\text{d}q_{x}(\theta)\text{d}\nu(x).
\]
if, and only if,
\[
  \int_{\X}\Bq(f)\,\text{d}\nu = \int_{\X} f \,d\nu, \; \forall f \in C(X,\mathbb{R}).
\]
Which is equivalent to $\Lq(\nu) = \nu$. Since any disintegration $\hat{\mu}$ has a marginal $\mu_{x}=q_x$, $\nu-a.e.$, this Borel probability measure $\hat{\mu}$ on $\Omega$, will be called
a {\bf holonomic lifting} of $\nu$, with respect to $\Rq$.

\section{Entropy and Pressure for IFSm}\label{entropy-and-pressure}

We now define two concepts of entropy, compare then, show sufficient conditions for them to be equal and introduce the topological pressure of a given potential, as well as the concept of equilibrium states. We show in this section a first result  on the existence of equilibrium states. In all that follows, the a priori measure has a special role (see \cite{lopes2015entropy}).

As in the previous section the mapping $T:\Omega\to X$ denotes the projection
on the first coordinate. Even when not explicitly mentioned,
any disintegrations of a probability measure $\hat{\nu}$, defined over $\Omega$,
will be from now on considered with respect to $T$.

\begin{defi}[Variational Entropy]\label{entropy}
Let $\mcal{R}$ be an IFS, $\hat{\nu} \in \HR$, $\mu$ a probability on $\Theta$ and
$d\hat{\nu}(x,\theta)=d\nu_{x}(\theta)d\nu(x)$ a disintegration of $\hat{\nu}$, with
respect to $T$.
The variational entropy of $\hat{\nu}$ with a priori probability $\mu$ is defined by
\[
  \hv(\hat{\nu})
   \equiv \inf_{ \substack{g\,\in\,C(\X, \R) \\ g>0 }}
        \left\{ \int_{\X} \ln\frac{\Bu(g)}{ g }  \,\text{d}\nu \right\}.
\]
\end{defi}

\begin{defi}
  When $\q={(\q_{x})}_{x\in\X}$ is a family of measures on $\Theta$ and $\mu$ a probability on $\Theta$,
  and $\nu$ is a probability on $\X$, we say that the family $\q$ is $\nu$-almost everywhere absolutely continuous with respect to $\mu$ when $\q_{x}\ll\mu$ for $\nu$-almost everywhere $x$ on $\X$ and write $\q\ll_{\nu}\mu$.

  If $\q\ll_{\nu}\mu$, we define $J_{x}(\theta)$ such that $J_{x} = \frac{\dq_{x}}{\dmu}$ when $\q_{x}\ll\mu$ and $J_{x}(\theta) = 0$ otherwise.
\end{defi}

\begin{defi}[Average entropy]
Let $\mcal{R}$ be an IFS, $\hat{\nu} \in \HR$,
$d\hat{\nu}(x,\theta)=d\nu_{x}(\theta)d\nu(x)$ a disintegration of $\hat{\nu}$ with
respect to $T$ and $\mu$ a probability on $\Theta$ such that $(\nu_{x})\ll_{\nu}\mu$.
The average entropy of $\hat{\nu}$ with respect to $\mu$ is defined by

\[
\ha(\hat{\nu})
  \equiv
  -\int_{\Omega} \ln J_{x}(\theta)\dnu_{x}(\theta)\dnu(x).
\]
\end{defi}

\begin{defi}[Optimal Function]\label{opt_function}
  Let $\mcal{R}$ be an IFS, $\hat{\nu}\in\HR$, $\text{d}\hat{\nu}(x,\theta)=\text{d}\nu_{x}(\theta)\,\text{d}\nu(x)$ a disintegration of $\hat{\nu}$ with respect to $T$ and $\mu$ a probability on $\Theta$ such that $(\nu_{x})\ll_{\nu}\mu$. We say that a positive function $g\in C(\X,\R)$ is optimal, with respect to the system of measures associated to
  the desintegration $\text{d}\hat{\nu}(x,\theta)=\text{d}\nu_{x}(\theta)\,\text{d}\nu(x)$
   if we have
  \[
    J_{x}(\theta) = \frac{g(\tau_{\theta}(x))}{\Bu(g)(x)}.
  \]
\end{defi}

\begin{prop}\label{entropy inequality ln}
  Let $\dq_{x}(\theta) = Q_{x}(\theta)\dmu(\theta)$ and $dp_{x}(\theta) = P_{x}(\theta)\dmu(\theta)$ be probabilities, and suppose $P$ is continuous, positive and bounded away from zero, while $Q$ is a positive  function which is integrable with respect to $\dq_{x}(\theta)\dnu(x)$. Then we have

  \[-\int_{\X}\int_{\Theta}\log(Q_{x}(\theta))\dq_{x}(\theta)\dnu(x) \le -\int_{\X}\int_{\Theta}\log(P_{x}(\theta)) \dq_{x}(\theta)\dnu(x).\]
\end{prop}
\begin{proof}
  Using Jensen's Inequality on $f(x) = -x \log(x)$, which is a  concave function, we have
  \begin{align*}
    \int_{\X}\int_{\Theta}-\log&\left(\frac{Q_{x}(\theta)}{P_{x}(\theta)}\right) \frac{Q_{x}(\theta)}{P_{x}(\theta)}\dpp_{x}(\theta)\dnu(x)\\
    &=\int_{\Omega} f\left(\frac{Q_{x}(\theta)}{P_{x}(\theta)}\right) \dpp_{x}(\theta)\dnu(x)\\
    &\le f\left(\int_{\Omega}\dpp_{x}(\theta)\dnu(x)\right) = f(1) = 0.
  \end{align*}
  Then,
  \begin{align*}
    \int_{\X}\int_{\Theta}&-\log\left(\frac{Q_{x}(\theta)}{P_{x}(\theta)}\right) \frac{Q_{x}(\theta)}{P_{x}(\theta)} P_{x}(\theta)\dmu(\theta)\dnu(x)\\ &= \int_{\X}\int_{\Theta}-\log\left(\frac{Q_{x}(\theta)}{P_{x}(\theta)}\right) \dq_{x}(\theta)\dnu(x) \le 0.
  \end{align*}
  Therefore,
  \[
    -\int_{\X}\int_{\Theta}\log\left(Q_{x}(\theta)\right) \dq_{x}(\theta)\dnu(x) \le     -\int_{\X}\int_{\Theta}\log\left(P_{x}(\theta)\right) \dq_{x}(\theta)\dnu(x).
  \]
\end{proof}

\begin{theorem}\label{entropy inequality ha hv}
Let $\mcal{R}$ be an IFS, $\hat{\nu} \in \HR$,
$\text{d}\hat{\nu}(x,\theta)=\text{d}\nu_{x}(\theta)\,\text{d}\nu(x)$ a disintegration of $\hat{\nu}$ with
respect to $T$,
 and $\mu$ a probability on $\Theta$ such that $(\nu_{x})\ll_{\nu}\mu$. Then
$$\ha(\hat{\nu}) \leq \hv(\hat{\nu}) \leq 0.$$
Moreover, if there exists some optimal function $\phi$, with respect to $\mathcal{R}_{ {q}}$,
then
\[
    \ha(\hat{\nu})
  = \hv(\hat{\nu})
  = \int_{\X} \ln \frac{\Bu(\phi)}{\phi} \dnu.
\]

\end{theorem}

\begin{proof}
	From the definition of variational entropy we obtain
	\[
	\hv(\hat{\nu})
	= \inf_{ \substack{ g\in C(\X, \R) \\ g>0  } }
    \left\{ \int_{\X} \ln \frac{\Bu(g)}{ g } \dnu \right\}
  \leq \int_{\X} \ln \frac{\Bu(1)}{1}\dnu = 0.
	\]

  It remains to show that $\ha(\hat{\nu}) \leq \hv(\hat{\nu})$.
  Let $g:X\to\R$ be continuous positive function and
  define for each $x\in X$  a probability where $\dpp_{x}(\theta)=g(\tau_{\theta}(x))/\Bu(g)(x)\dmu(\theta)$.
  From Proposition $\ref{entropy inequality ln}$
  and the properties of the holonomic measures
  we get the following inequalities for any continuous and positive function $g$:
    \begin{align*}\label{desigualdade-ha-otimal}
    \ha(\hat{\nu})
      &= -\int_{\X}\int_{\Theta} \ln J_{x}(\theta)\dq_{x}(\theta)\dnu(x)
      \\[0.2cm]
      &\leq -\int_{\X}\int_{\Theta} \ln\left( \frac{g\circ\tau_{\theta}}{\Bu(g)}\right) \dq_{x}(\theta)\dnu(x)
      \\[0.2cm]
      &= -\int_{\X}
        \left[
          \int_{\Theta} \ln(g\circ\tau_{\theta})\dq_{x}(\theta) -
          \int_{\Theta} \ln(\Bu(g))\dq_{x}(\theta)
        \right]
        \dnu(x)
      \\[0.2cm]
      &= -\int_{\X} \Bq(\ln g) \dnu +
         \int_{\X} \ln (\Bu(g)) \dnu
      \\
      &= -\int_{\X}\ln g\dnu + \int_{\X}\ln (\Bu(g))\dnu
      \\[0.3cm]
      &= \int_{\X} \ln \frac{\Bu(g)}{g}\dnu,
    \end{align*}

    Therefore, $\ha(\hat{\nu}) \le \hv(\hat{\nu})$. Furthermore, if $J_{x}(\theta) = \phi \circ \tau_{\theta}(x) / \Bu(\phi)(x)$ for some $\phi > 0$ continuous function, then
    $\ha(\hat{\nu}) = \int_{\X} \log \frac{\Bu(\phi)}{\phi} = \hv(\hat{\nu})$.
\end{proof}

\begin{remark}
   We would like to address a very important distinction between $\hv(\hat{\nu})$ and $\ha(\hat{\nu})$. The first one, the variational entropy, is to be used at the variational principle and to define the equilibrium states having no additional requirements except that $\mu$ is a probability. On the other hand the average entropy is only defined for holonomic probabilities having a marginal absolutely continuous with respect to $\mu$ and will not be used in a variational principle. The only role of this quantity is to be a lower bound to the variational entropy, when it does exist.
\end{remark}

\begin{defi}\label{pressure definition}
  Let $\psi: \X \to \R$ be a  positive continuous function, $\mu$ a probability on $\Theta$, and $\mcal{R}_{\psi}(\X, \tau, \q)$ an IFSm, where $\dq_{x}(\theta) = \psi\circ\tau_{\theta}(x)\dmu(\theta)$.
  The topological pressure of $\psi$, with respect to $\mcal{R}_{\psi}$, is defined by
  \begin{equation}\label{pressure sup inf}
    P(\psi) = \sup_{\hat{\nu}\in\HR}\inf_{g \in C(\X;\R)\\g > 0}\left\{\int_{\X}\ln\frac{\Bq(g)}{g}\dnu\right\},
  \end{equation}
  where $\nu := \pi_{*}\hat{\nu}$ for $\pi:\Omega\to\X$ the $\X$ projection.
\end{defi}

\begin{lemma}\label{pressure alternative sup}
  Let $\psi: \X \to \R$ be a  positive continuous function and $\mcal{R}_{\psi} = (\X, \tau, \q)$ be the $IFSm$ defined above, where $\dq_{x}(\theta) = \psi\circ\tau_{\theta}(x)\dmu(\theta)$. Then, the topological pressure of $\psi$ is alternatively given by

  \begin{equation*}
    P(\psi) = \sup_{\hat{\nu}\in\HR} \left\{\hv(\hat{\nu}) + \int_{\X}\log \psi \dnu\right\}.
  \end{equation*}
\end{lemma}
\begin{proof}
  First, note that $\Bq(g) = \Bu(g\cdot\psi)$ where $(g\cdot\psi)(x)=g(x)\psi(x)$. In fact
  \[
    \Bq(g)(x) = \int_{\Theta} g\circ\tau_{\theta}(x) \psi\circ\tau_{\theta}(x)\dmu(\theta) = \int_{\Theta} (g\cdot\psi)\circ\tau_{\theta}(x)\dmu(\theta) = \Bu(g\cdot\psi)(x).
  \]

  To finish the proof, we only need to use the pressure's definition and some basic properties as follows:

  \begin{align*}
    P(\psi) &= \sup_{\hat{\nu}\in\HR}\inf_{g > 0}\left\{\int_{\X}\ln\frac{\Bq(g)}{g}\dnu\right\},\\
    &= \sup_{\hat{\nu}\in\HR}\inf_{g > 0}\left\{\int_{\X}\log \psi \dnu - \int_{\X}\log \psi \dnu + \int_{\X}\ln\frac{\Bq(g)}{g}\dnu\right\}\\
    &= \sup_{\hat{\nu}\in\HR}\inf_{g > 0}\left\{\int_{\X}\log \psi \dnu +  \int_{\X}\ln\frac{\Bu(g\cdot\psi)}{g\cdot\psi}\dnu\right\}\\
    &= \sup_{\hat{\nu}\in\HR}\left\{\int_{\X}\log \psi \dnu +  \inf_{\tilde{g} > 0}\int_{\X}\ln\frac{\Bu(\tilde{g})}{\tilde{g}}\dnu\right\}\\
    &= \sup_{\hat{\nu}\in\HR}\left\{\hv(\hat{\nu}) + \int_{\X}\log \psi \dnu\right\}.
  \end{align*}
\end{proof}

\begin{remark}
Note that if $\dq_{x}(\theta) = \frac{\psi\circ\tau_{\theta}(x)}{\Bu(\psi)(x)}\dmu(\theta)$, by the Theorem $\ref{prop-exist-auto-medida}$ there exists $\rho > 0$ and $\nu$ s.t. $\Lq(\nu) = \rho\nu$.

But,
\begin{align*}
  \rho &= \int_{\X}\text{d}\Lq(\nu) = \int_{\X} \Bq(1)(x) \dnu(x)= \int_{\Omega}\frac{\psi\circ\tau_{\theta}(x)}{\Bu(\psi)(x)}\dmu(\theta)\dnu(x)\\
  &=\int_{\X}{\Bu(\psi)(x)}^{-1}\int_{\Theta}\psi\circ\tau_{\theta}(x)\dmu(\theta)\dnu(x) = \int_{\X}\dnu = 1.
\end{align*}

Therefore we have
\[P(\psi) \ge \sup_{\{\nu\, |\, \Lq(\nu) = \nu\}}\int_{\X}\ln B_{\mu}(\psi) \dnu.\]
\end{remark}

\begin{defi}[Equilibrium States]
  Let $\mcal{R}$ be an IFS, $\hat{\nu}\in\HR$ and $\mu$   a probability on $\Theta$.
  Let $\psi: \X \to \R$ be a  positive continuous function. We say that the holonomic measure
  $\hat{\nu}$ is an equilibrium state for $(\psi,\mu)$ if
  \[\hv(\hat{\nu}) + \int_{\X}\log \psi \dnu = P(\psi).\]
\end{defi}

\begin{lemma}
Let $\X$ and $\Y$ be compact metric spaces and $T:\Y\to \X$ be a continuous mapping.
Then the push-forward mapping
$\Phi_{T}\equiv \Phi:\mathscr{M}_{1}(\Y)\to \mathscr{M}_{1}(\X)$ given by
\[
\Phi(\hat{\mu})(A)= \hat{\mu}(T^{-1}(A)),
\quad\text{where}\ \hat{\mu}\in \mathscr{M}_{1}(\Y)\ \text{and} \ A\in
\mathscr{B}(\X)
\]
is weak-$*$ to weak-$*$ continuous.
\end{lemma}
\begin{proof}
Since we are assuming that $\X$ and $\Y$ are  compact metric spaces then
we can ensure that the weak-$*$ topology of both
$\mathscr{M}_{1}(\Y)$ and $\mathscr{M}_{1}(\X)$
are metrizable. Therefore is enough to prove that $\Phi$ is sequentially continuous.
Let $(\hat{\mu}_n)_{n\in\mathbb{N}}$ be a sequence
in $\mathscr{M}_{1}(\Y)$ so that $\hat{\mu}_n\rightharpoonup \hat{\mu}$.
For any continuous real function $f:\X\to\mathbb{R}$
we have from change of variables theorem that
\[
\int_{\X} f\, d[\Phi(\hat{\mu}_n)]
=
\int_{\Y} f\circ T\, d\hat{\mu}_n,
\]
for any $n\in\mathbb{N}$. From the definition
of the weak-$*$ topology, it follows that the right hand side above
converges when $n\to\infty$, and we have
\[
\lim_{n\to\infty}\int_{\X} f\, d[\Phi(\hat{\mu}_n)]
=
\lim_{n\to\infty}\int_{\Y} f\circ T\, d\hat{\mu}_n
=
\int_{\Y} f\circ T\, d\hat{\mu}
=
\int_{\X} f\, d[\Phi(\hat{\mu})].
\]
The last equality shows that $\Phi(\hat{\mu}_n)\rightharpoonup \Phi(\hat{\mu})$
and consequently the weak-$*$ to weak-$*$ continuity of $\Phi$.
\end{proof}

For any $\hat{\nu}\in\mathcal{H}(\mathcal{R})$ it is always possible to disintegrate
it as $d\hat{\nu}(x,i)= d\nu_{x}(i)d[\Phi(\hat{\nu})](x)$, where $\Phi(\hat{\nu})\equiv \nu$ is
the probability measure on $\mathscr{B}(\X)$, defined for any
$A\in\mathscr{B}(\X)$ by
\begin{align}\label{def-Phi}
\nu(A)\equiv \Phi(\hat{\nu})(A) \equiv  \hat{\nu}(T^{-1}(A)),
\end{align}
where
$T:\Omega\to \X$ is the canonical projection of the first coordinate.
This observation together with the previous lemma allow us to
define a continuous mapping from $\mathcal{H}(\mathcal{R})$
to $\mathscr{M}_{1}(\X)$ given by $\hat{\nu}\longmapsto \Phi(\hat{\nu})\equiv \nu$.

We now prove a theorem ensuring the existence of equilibrium states for
any continuous positive function $\psi$.
Although this theorem has clear and elegant proof and works
in great generality it has the disadvantage of
providing no description of the set of equilibrium states.

\begin{theorem}[Existence of Equilibrium States]
	Let $\mathcal{R}$ be an IFS, $\psi:X\to\mathbb{R}$ a positive continuous function and $\mu$ a probability on $\Theta$.
	Then the set of equilibrium states for $(\psi,\mu)$ is not empty.
\end{theorem}
\begin{proof}
As we observed above we can define a weak-$*$ to weak-$*$ continuous mapping
\[
\mathcal{H}(\mathcal{R})\ni \hat{\nu}\longmapsto \nu\in
\mathscr{M}_{1}(\X),
\]
where $d\hat{\nu}(x,i)=d\nu_{x}(i)d\nu(x)$ is the above constructed
disintegration of $\hat{\nu}$.
From this observation, it follows that
for any fixed positive continuous $g$ the mapping
$
\mathcal{H}(\mathcal{R})\ni\hat{\nu}
\longmapsto
\int_\X \ln (B_{\mu}(g)/ g) \, d\nu
$
is continuous with respect to the weak-$*$ topology. Therefore
the mapping
\[
\mathcal{H}(\mathcal{R})\ni\hat{\nu}
\longmapsto
\inf_{ \substack{ g\in C(\X, \mathbb{R}) \\ g>0  } }
\left\{ \int_{\X} \ln \frac{B_{\mu}(g)}{g}\,  d\nu \right\}
\equiv
\hv(\hat{\nu}).
\]
is upper semi-continuous (USC) which implies by standard results that
the following mapping is also USC
\[
\mathcal{H}(\mathcal{R})\ni\hat{\nu}
\longmapsto
\hv(\hat{\nu})+  \int_{\X} \ln( \psi(x))\, d\nu(x).
\]
Since $\mathcal{H}(\mathcal{R})$ is compact in the weak-$*$ topology and
the above mapping is USC,  it follows that
this mapping attains its supremum at some $\hat{\mu}\in \mathcal{H}(\mathcal{R})$, i.e.,
\[
\sup_{\hat{\nu} \in \mathcal{H}(\mathcal{R})}
\left\{ \hv(\hat{\nu}) +  \int_{\X} \ln \psi \, d\nu \right\}
=\hv(\hat{\mu})
+  \int_{\X} \ln \psi \, d\mu
\]
thus proving the existence of at least one equilibrium state.
\end{proof}

Let $\psi: \X \to \R$ be a positive continuous function, $\mu$ a probability on $\Omega$, and $\mcal{R}_{\psi}(\X, \tau, \q)$ an IFSm, where $\dq_{x}(\theta) = \psi\circ\tau_{\theta}(x)\dmu(\theta)$.
Suppose that there is $h$ a positive continuous function such that
$\Bq(h) = \Bu(h\cdot\psi) = \rho(\Bq)h$.
Then we can define, following Definition $\ref{opt_function}$, $\Rq[\p]=(\X,\tau,\p)$ where
\[\frac{\dpp_{x}}{\dmu}(\theta) := \frac{(h\cdot\psi)\circ\tau_{\theta}(x)}{\Bu(h\cdot\psi)} = \frac{(h\cdot\psi)\circ\tau_{\theta}(x)}{\rho(\Bq)h} = \frac{h\circ\tau_{\theta}(x)}{\rho(\Bq)h}\cdot\frac{dq_{x}}{\dmu}(\theta).\]

The IFSm $\Rq[\p]$ is called the normalization of $\Rq$. Take $\Lq[\p](\nu) = \nu$ and let $\hat{\nu}$ be the holonomic lifting of $\nu$. Then by the Theorem $\ref{entropy inequality ha hv}$ we know that

\[\hv(\hat{\nu}) = \int_{\X}\log\frac{\Bu(h\cdot\psi)}{h\cdot\psi}\dnu = \log\rho(\Bq) - \int_{\X}\log\psi\dnu.\]

Then, choosing this $\hat{\nu}$ as particular in supremum given in Lema $\ref{pressure alternative sup}$, $P(\psi) \ge \hv(\hat{\nu}) + \int_{\X}\log\psi\dnu = \log\rho(\Bq)$.

But, remember that the pressure, defined in expression $(\ref{pressure sup inf})$, is
\[
    P(\psi) = \sup_{\hat{\nu}\in\HR}\inf_{g \in C(\X;\R)\\
    g > 0}\left\{\int_{\X}\ln\frac{\Bq(g)}{g}\dnu\right\},
\]

then,
\[\inf_{g \in C(\X;\R)\\g > 0}\left\{\int_{\X}\ln\frac{\Bq(g)}{g}\dnu\right\} \le \int_{\X}\ln\frac{\Bq(h)}{h}\dnu = \log\rho(\Bq).\]

Taking the supremum over $\HR$ in both sides of above inequality, we have
$P(\psi)\le\log\rho(\Bq)$. Since the reverse inequality is already shown, we prove that
\[P(\psi) = \log\rho(\Bq).\]

\begin{remark}
We end this section by showing that the $IFSm$ Thermodynamic Formalism generalizes the Thermodynamical Formalism for a dynamical system.
Let $\Theta$ be a compact metric space and $\X = \Theta^\N$.
For each $\theta\in\Theta$ define $\sigma_\theta(x_1, x_2, \ldots) =
(\theta, x_1, x_2, \ldots)$ the inverse branch of the right shift $\sigma$.
Take $\mu$ a a-priori probability on $\Theta$. Let $\psi: \Omega\to\R$ be
a positive potential and $A = \log\psi$.
Now we define $\dq_x(\theta) = e^{A\circ\sigma_\theta(x)}\dmu(\theta)$.
We have
\[
\Bq(\phi) = \int_S e^{A\circ\sigma_\theta(x)}\phi\circ\sigma_\theta(x)\dmu(\theta)
= \Ruelle_A(\phi)(x),
\]
where $\Ruelle_A$ is the Ruelle Operator for the right shift $\sigma$ and
the potential $A$ (see~\cite{lopes2015entropy} for more details).
By Definition~\ref{pressure definition},
\begin{align*}
	P(\psi) &= \sup_{\hat{\nu}\in\HR}
	\inf_{g > 0}
	\left\{\int_{\X}\ln\frac{\Bq(g)}{g}\dnu\right\}\\
	&= \sup_{\nu\in\mathcal{M}_\sigma}
	\inf_{g > 0}
	\left\{\int_{\X}\ln\frac{\Ruelle_{A}(g)}{g}\dnu\right\}.
\end{align*}
The last expression is exactly the pressure of the potential $A$
in Thermodynamical Formalism.
Suppose that there is $\phi_A$ a positive continuous function such that
$\Bq(\phi_A) = \Ruelle_{A}(\phi_A) = \rho(\mathcal{R}_{A})\phi_A = \lambda_A \phi_A$.
Then $P(e^A) = \log\lambda_A$. For instance, we know that if $A$ is H\"older,
then there exists such $\phi_A > 0$ function.
From this example, we can see that the $IFSm$ Thermodynamic Formalism,
in certain sense, generalizes the Thermodynamical Formalism for a dynamical
system. When we look at the family $\{\sigma_\theta\}_{\theta\in\Theta}$
of functions, we are looking at the inverse branches of the dynamical system.
\end{remark}

%%%%%%%%%%%%%%%%%%%%%%%%%%%%%%%%%%%%%%%%%%%%%%%%%%%%%%%%%%%%%%%%%%%%%%%%%%

\section{Pressure Differentiability and Equilibrium States}\label{uniqueness-for-eqstates}

We show in this section a uniqueness result for the equilibrium state introduced in the last section.
In order to do that we will need to consider the functional $p:C(\X,\mathbb{R})\to \mathbb{R}$
given by
\begin{align}\label{def-funcional-p}
p(\varphi) = P(\exp(\varphi)).
\end{align}
It is immediate to verify
that $p$ is a convex and finite valued functional.
We say that a Borel signed measure $\nu\in\mathscr{M}_{s}(X)$ is a
{\bf subgradient} of $p$ at $\varphi$ if it satisfies the following
subgradient inequality
$$p(\eta)\geq p(\varphi)+\nu(\eta-\varphi) \;\; \mbox{for any}\;\; \eta \in \mathscr{M}_{s}(X).
$$
The set of all subgradients at $\varphi$ is called {\bf subdifferential} of $p$
at $\varphi$ and denoted by $\partial p(\varphi)$.
It is well-known that if $p$ is a continuous mapping
then $\partial p(\varphi)\neq \emptyset$ for any $\varphi\in C(\X,\mathbb{R})$.

We observe that for any pair $\varphi,\eta\in C(\X,\mathbb{R})$
and $0<t<s$, it follows from the convexity of $p$ the following inequality
$$
s( p(\varphi+t\eta)-p(\varphi))\leq t(p(\varphi+s\eta)-p(\varphi)).
$$
In particular, the one-sided directional derivative
$d^{+}p(\varphi):C(\X,\mathbb{R})\to \mathbb{R}$ given by
\[
d^{+}p(\varphi)(\eta)
=
\lim_{t\downarrow 0} \frac{p(\varphi+t\eta)-p(\varphi)}{t}
\]
is well-defined for any $\varphi\in C(\X,\mathbb{R})$.

\begin{theorem} \label{teo-gateaux-unicidade}
For any fixed $\varphi\in C(\X,\mathbb{R})$ we have
\begin{enumerate}
\item
the signed measure $\nu\in\partial p(\varphi)$ iff
$\nu(\eta)\leq d^{+}p(\varphi)(\eta)$ for all $\eta\in C(X,\mathbb{R})$;

\item
the set $\partial p(\varphi)$ is a singleton iff $d^{+}p(\varphi)$
is the G\^ateaux derivative of $p$ at $\varphi$.
\end{enumerate}
\end{theorem}
\begin{proof}
This theorem is a consequence of Theorem 7.16 and Corollary 7.17
of the reference~\cite{charalambos2006infinite}.
\end{proof}

\begin{corollary}
Let $\mcal{R}$ be an IFS, $\psi:\X\to\mathbb{R}$ a positive continuous function, $\mu$ a probability on $\Theta$
and $\Phi(\hat{\nu}) = \nu$ for $\hat{\nu}\in\HR$ where $\nu$ is given by disintegration with respect to $T$.
If the functional $p$ defined on $\eqref{def-funcional-p}$
is G\^ateaux differentiable at $\varphi\equiv \log \psi$,
then
\[
\# \{\Phi(\hat{\mu}):\ \hat{\mu}\ \text{is an equilibrium state for}\ \psi \} =1.
\]
\end{corollary}

\begin{proof}
Suppose that $\hat{\mu}$ is an equilibrium state for $\psi$. Then
we have from the definition of the pressure that
\begin{align*}
p(\varphi+t\eta) -p(\varphi)
&=
P(\psi\exp(t\eta))-P(\psi)
\\
&\geq
h_{v}(\hat{\mu})+\int_{\X} \ln \psi\ d\mu + \int_{\X} t\eta\, d\mu
-
h_{v}(\hat{\mu})-\int_{\X} \ln \psi\ d\mu
\\
&=
t\int_{\X} \eta\, d\mu.
\end{align*}
Since we are assuming that $p$ is G\^ateaux differentiable at $\varphi$, it
follows from the above inequality that $\mu(\eta)\leq d^{+}p(\varphi)(\eta)$
for all $\eta\in C(\X,\mathbb{R})$. From this inequality and
Theorem \ref{teo-gateaux-unicidade}
we can conclude  that $\partial p(\varphi)=\{\mu\}$.
Therefore for all equilibrium state $\hat{\mu}$ for $\psi$
we have $\Phi(\hat{\mu})=\partial p(\varphi)$, thus finishing the proof.
\end{proof}

%%%%%%%%%%%%%%%%%%%%%%%%%%%%%%%%%%%%%%%%%%%%%%%%%%%%%%%%%%%%%%%%%%%%%%%%
\section{Possible Application in Economics}\label{application_economics}

In Gupta et al.~\cite{gupta2021study} a chaos game is used to represent a time series as a PC plot and compare similarities and dissimilarities in different  time frame such as the global pandemic of COVID-19. More precisely, the author consider the set $X=[0,1]^2$ as the base space and the four linear contractions
\begin{equation}\label{eq:IFS_gupta}
	\left\{
	\begin{array}{l}
		\tau_{A}(x,y)=(0.5 x ,0.5 y )  \\
		\tau_{B}(x,y)=(0.5 x+0.5,0.5 y ) \\
		\tau_{C}(x,y)=(0.5 x ,0.5 y+0.5) \\
		\tau_{D}(x,y)=(0.5 x+0.5,0.5 y+0.5)
	\end{array}
	\right.
\end{equation}
so $(X, \tau_\theta)_{\theta \in \Theta}, \Theta=\{A,B,C,D\}$, is a classic contractive system whose attractor is $X$ itself. Consider the identification:\\
A – if the market falls more than 0.01\% of the previous value,\\
B – if the market falls less than 0.01\% of the previous value, \\
C – if the market gains less than 0.01\% of the previous value and \\
D – if the market gains more than 0.01\% of the previous value,\\
in this way the time series of length $N$ associated to a certain economic indicator is translated in to a genetic sequence
$$\gamma=(DACCADCDACDC...AACCBADD) \in \Theta^N.$$
Fixed an arbitrary initial point $Z_0=(x_0,y_0)=(0.5, 0.5)$ the chaos game consists in iterating $(x_0,y_0)$ by each map $Z_1=(x_1,y_1)=\tau_{D}(x_0,y_0), Z_2=(x_2,y_2)=\tau_{A}(x_1,y_1), Z_3=(x_3,y_3)=\tau_{C}(x_2,y_2),....$ according to $\gamma$. Considering $M\geq 2$ and the diadic partition of $X$ given by $$\bigcup_{\gamma' \in \Theta^M} \tau_{\gamma'_{M}}( \cdots (\tau_{\gamma'_{1}}(X))),$$
the PC plot $W$ is a grey scale picture where the color of the each individual part $\Lambda=\tau_{\gamma'_{M}}( \cdots (\tau_{\gamma'_{1}}(X)))$ is the frequency of visits of the chaos game orbit $\{Z_j, j\geq 0\}$ to $\Lambda$ that is,
$$W(\Lambda)= \frac{1}{N} \, \sharp \{j=0,...,N-1\, | \, Z_j \in \Lambda\} \sim \mu(\Lambda).$$
Obviously, $\nu_{N}=\sum_{\Lambda} W(\Lambda) \delta_{(x_{\Lambda}, \, y_{\Lambda})}$, where $(x_{\Lambda}, \, y_{\Lambda}) \in \Lambda$ is arbitrary, is a discrete probability and, if $\mu(\partial(\Lambda))=0$ then by the EET (\cite{elton1987ergodic}, Corollary 2), when $N \to \infty$,  $\nu_{N}$ converge in distribution to the invariant measure $\mu$ for the IFS with probabilities $(X, \tau_\theta, p_\theta)_{\theta \in \Theta}$, where $p_A, p_B, p_C, p_D$  are the relative frequency of each symbol $A, B, C, D$ in $\gamma$, respectively.\\
For instance, if $N=100$ and if a certain time series produces the genetic sequence $$\gamma=( A ,  A ,  D ,  D ,  A ,  ... , A ,  D ,  B ,  C ,  C ,  B ,  A , D) \in \{A,B,C,D\}^{100},$$
we obtain the frequencies $[p_A, p_B, p_C, p_C]=[0.39, \,0.17, \,0.15,\, 0.29]$,
and considering $M=4$ we obtain the following PC plot
\begin{figure}[!ht]
	\centering
	\includegraphics[width=6cm ]{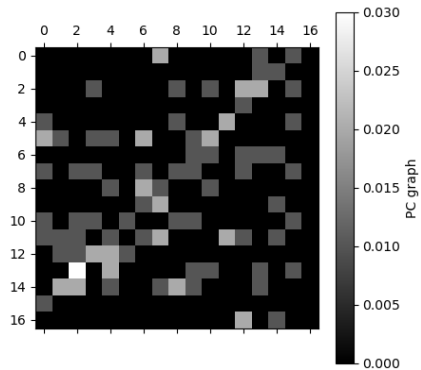}
	\caption{PC plot where each square represents one element $\Lambda$ of the diadic partition and grey scale value $0 \leq \frac{1}{N} \, \sharp \{j=0,...,N-1\, | \, Z_j \in \Lambda\} \leq 1$.}\label{fig:PCplot}
\end{figure}
which is an approximation for the invariant measure $\mu$ of the associated IFS with probabilities $[0.39, \,0.17, \,0.15,\, 0.29]$.

In order to generalize this idea we need to consider an infinite compact continuous range of values of the economic indicator, such as $\Theta=[0\%, 100\%]$, instead of taking only four values $\Theta=\{A,B,C,D\}$. Also, it is not reasonable to suppose that the probability of a change of $\theta\%$ in the indicator is independent of the current state of the indicator: the distribution of the occurrence of  $\theta \in [0\%, 100\%]$, given the current state $Z \in X$ must be a measure of probability $q_{Z}( \cdot)$ over $[0\%, 100\%]$. Therefore, we believe the theory developed in the previous sections should be used when making this generalization.

{\bf Aknowledgement:} \emph{We would like to thanks Prof. Leandro Cioletti whose contribution on the related preprint \cite{CiolOLivArxiv2017} was fundamental to establish the tools and ideas that we now generalize in our work.}
%%%%%%%%%%%%%%%%%%%%%%%%%%%%%%%%%%%%%%%%%%%%%%%%%%%%%%%%%%%%%%%%%%%%%%%%

\end{document}